\documentclass[12pt]{amsart}

\usepackage{amssymb}
\usepackage{graphicx}

\newtheorem{theorem}{Theorem}

\newtheorem*{theorem*}{Theorem}

\newcommand{\tn}{|\kern-.10em|\kern-.10em|}

\def\om{\Omega}

\def\({\big(}
\def\){\big)}

\def\Qbar{\hbox{\sl Q\kern-.45em{\vrule height.63em width.05em
depth-.033em}}~}
\def\qbar{{{\scriptstyle Q}\kern-.45em{\vrule height.41em width.035em
depth-.03em}}~}
\def\Cbar{\hbox{\sl C\kern-.35em{\vrule height.63em width.05em
depth-.033em}}~}
\def\cbar{{{\scriptstyle C}\kern-.41em{\vrule height.42em width.035em
depth-.03em}}~}
\def\ibid{\hbox to.5truein{\hrulefill}}
\def\IN{\hbox{{\rm I}\kern-.13em{\rm N}}}

\def\IR{\hbox{{\rm I}\kern-.13em{\rm R}}}

\def\vvs{\vspace{2\jot}}

\def\hdotfill{\leaders\hbox to 1em{\hss.\hss}\hfill}
\def\hrf{\hbox to.75in{\hrulefill}}

\newcommand{\clsp}{{\overline{\textnormal{span}}}^{w^*}}
\newcommand{\kve}{\mathbb{Q}}

\newcommand\rls{\mathbb R}

\calclayout

\def\om{\omega_1}

\begin{document}

\title{Mazur Intersection Property For Asplund Spaces}

\author{Miroslav Ba\v c\' ak}
\address{Department of Mathematical Analysis\\Charles University\\Sokolovsk\'a 83\\186 75 Praha 8\\Czech Republic}
\email{bacak@karlin.mff.cuni.cz}
\author{Petr H\'ajek}
\address{Mathematical Institute\\Czech Academy of Science\\\v Zitn\'a 25\\115 67 Praha 1\\Czech Republic}
\email{hajek@math.cas.cz}
\date{March 2008}
\thanks{Supported by grants: Institutional Research Plan AV0Z10190503, GA\v CR 201/07/0388, GA UK 72 407,
GA \v CR 201/07/0394}
\begin{abstract}
The main result of the present note states that
 it is consistent with the ZFC axioms of set theory
(relying on  Martin's Maximum MM axiom), that every Asplund
space of density character $\om$ has a renorming with the Mazur intersection property.
Combined with the previous result of Jim\' enez and Moreno (based upon the work of Kunen under the continuum hypothesis)
 we obtain that the MIP renormability of Asplund spaces of
density $\om$ is undecidable in ZFC.

\end{abstract}

\maketitle

The Mazur intersection property (MIP for short)
 was first investigated by S. Mazur  in \cite{M} as a purely geometrical isometric property of
a Banach space, and has since been studied extensively over the years. An early result of Mazur  claims that
a Banach space with a Fr\' echet differentiable norm (necessarily an Asplund space)
has the MIP (\cite{M}). Phelps \cite{P}
proved that a separable Banach space has a MIP renorming if and only if its dual is separable, or equivalently, if it is
an Asplund space. Much of the further progress in the theory depended on an important characterization of MIP,
due to Giles, Gregory and Sims \cite{GGS}, by the property that $w^*$-denting points of $B_{X^*}$ are norm dense
in $S_{X^*}$.
This result again suggests a close connection of MIP to Asplund spaces, as the latter can be characterized in a similar way
as spaces such that bounded subsets of their dual are $w^*$-dentable.
It has opened a way to applying biorthogonal systems to the MIP.
Namely, Jim\' enez and Moreno \cite{JM} have proved that if a Banach space $X^*$ admits a fundamental biorthogonal system
$\{(x_\alpha, f_\alpha)\}$, where $f_\alpha$ belong to $X\subset X^{**}$, then $X$ has a MIP renorming.
As a corollary to this criterion (\cite{JM}), they get
 that every Banach space can be embedded into a Banach space which is MIP renormable,
a rather surprising result which in particular strongly shows that MIP and Asplund properties, although closely related,
are distinct. We remark that MIP has connections with other parts of Banach space theory, such as the ball topology
(\cite{gk}) etc.

On the other hand, all
 known examples of an Asplund space without MIP renorming (\cite{JM}, \cite{bos}), such as $C(K)$ space, where $K$ is
a compact space constructed by Kunen (\cite{N}), were constructed using additional set theoretical assumptions, such as
the continuum hypothesis CH (or $\clubsuit$). In fact $C(K)$, where $K$ is Kunen's compact, has density character $\om$,
the smallest density for which such a result may hold in view of Phelps' theorem.

\vvs

The main result of the present note goes in the positive direction. Namely,
 it is consistent with the ZFC axioms of set theory
(relying on a powerful additional axiom - Martin's Maximum MM \cite{fms}), that every Asplund
space of density character $\om$ has a renorming with the Mazur intersection property.
 We thank V. Zizler for suggesting the problem and sharing his insights with us.
Combined with the previous result of Jim\' enez and Moreno we obtain that the MIP renormability of Asplund spaces of
density $\om$ is undecidable in ZFC.

\vvs

Our proof has three main ingredients. It combines a recent result of Todor\v cevi\' c,
claiming that under MM every Banach space with density
$\om$ has a fundamental system, with the machinery of projectional generators, and a criterion of MIP renormability
due to Jim\' enez and Moreno via fundamental systems.
We recall that MM is an axiom consistent with
ZFC, see \cite{fms} for details. We prefer not to state MM explicitly, as we do not enter into any axiomatic arguments here,
 and our reliance on MM is solely
through the use of Todor\v cevi\' c's theorem in \cite{tod} (see also \cite{bos}).
In order to explain in some detail  the role of these results, we need some preliminary definitions.

\vvs

Let $(X,\|\cdot\|)$ be a Banach space, denote its closed unit ball $B_X$ and its unit sphere $S_X.$
More generally, $B(x,\rho)=\{y\in X: \|y-x\|\le\rho\}$ is a ball centered at $x$ with radius $\rho\ge0$.
 The dual of $X$ is denoted $X^*,$
and by $w^*$ we mean the weak$^*$ topology on $X^*,$ that is $\sigma(X^*,X).$
The symbol $\tau_p$ stands for the topology of pointwise convergence in $c_0(\om).$ The set of rational numbers and the linear hull with rational coefficients are denoted $\kve, \: \kve$--span, respectively.

A Banach space $X$ is an \textit{Asplund space} if every separable subspace of $X$ has a separable dual space. We say that
a Banach space $(X,\|\cdot\|)$ has the \textit{Mazur intersection property} if every bounded closed convex set $K$ is an
intersection of closed balls. More precisely, $K=\cap_\alpha B(x_\alpha,\rho_\alpha)$, for some system of points
$x_\alpha\in X$, and radii $\rho_\alpha$.

Let $X$ be a Banach space and $\Gamma$ a nonempty set. A family $\{(x_{\gamma},x_{\gamma}^* )\}_{\gamma \in \Gamma}$
of pairs  $x_\gamma\in X$, $x_\gamma^*\in X^*$ is called a \textit{biorthogonal system} in $X \times X^*$ if
$x_{\alpha}^*(x_{\beta})=\delta_{\alpha \beta},$ where $ \delta_{\alpha \beta}$ is Kronecker's delta, for all
$\alpha,\beta \in \Gamma.$ A biorthogonal system is \textit{fundamental} if
$\overline{\text{span}} \{x_{\gamma}\}_{\gamma \in \Gamma} = X.$

If $A$ is a subset of a Banach space $X$ and $B \subset X^*,$ we say
that $A$ \textit{countably supports} $B,$ (or that $B$ is \textit{countably supported} by $A$), provided
$\{ a \in A:\: b(a)\neq 0 \}$ is at most countable, for all $b\in B.$ A subset $N \subset X^*$ is \textit{1--norming} if
$\|x\|= \sup\{x^*(x):\: x^* \in N \cap B_{X^*}\},$ for all $x \in X.$

For a nonseparable Banach space $X$ of density $\mu$, we define \textit{projectional resolution of identity (PRI)} as a
family $\{P_{\alpha}:\: \omega \leq \alpha \leq \mu \}$ of linear projections on $X$ such that $P_{\omega}=0,$ $P_{\mu}$
is the identity mapping on $X,$ and for all $\omega \leq \alpha,\:\beta \leq \mu:$
\begin{enumerate}
 \item $\|P_{\alpha}\|=1,$
 \item $\text{dens} P_{\alpha} X \leq \text{card} \alpha,$
 \item $P_{\alpha}P_{\beta}=P_{\beta}P_{\alpha} = P_{\alpha},$ if $\alpha \leq \beta,$
 \item $\cup_{\alpha<\beta} P_{\alpha+1} X$ is norm dense in $P_{\beta}X,$ if $\omega<\beta.$
\end{enumerate}

A standard way of obtaining PRI on Banach spaces is using the  Projectional Generator. This technique has an
advantage of allowing some additional properties for the PRI, important in our proof.

Let $X$ be a Banach space and $W \subset X^*$ be 1--norming set closed under linear combinations with rational coefficients.
 Let $\Phi :W \rightarrow 2^{X}$ be an at most countably valued mapping such that for every nonempty set $B \subset W$
with linear closure, $\Phi(B)^{\perp} \cap \overline{B}^{w^*} = \{ 0 \}.$ Then the couple $(W,\Phi)$ is called a
\textit{projectional generator (PG).}

Let $\{P_{\alpha}:\: \omega \leq \alpha \leq \mu \}$ be a PRI on $X$, $G\subset X$. We will say that the given PRI is
subordinated to the set $G$ if $P_\alpha(x)\in\{0,x\}$ for all $\alpha$ and $x\in G$.

We are now going to present the main theorems, whose combination leads to the proof of our main result.

\begin{theorem}\cite[Theorem 3.42]{bos}

Let $X$ be a nonseparable Banach space with a projectional generator $(W,\Phi)$, and a set $G\subset X$ that countably
supports $W$. Then $X$ has a PRI $\{P_{\alpha}:\: \omega \leq \alpha \leq \mu \}$, which is subordinated to the set $G$.
\label{subo}
\end{theorem}

We shall use the following theorem of Jim\' enez and Moreno (\cite{JM}, see also \cite[Theorem 8.42]{bos}).
\begin{theorem}
 Let $X^*$ be a dual Banach space with a fundamental biorthogonal system
$\{(x_{\alpha}^*,x_{\alpha})\}_{\alpha < \om} \subset X^* \times X^{**}$, with the property that
$x_\alpha\in X\subset X^{**}$. Then $X$ admits an equivalent norm with the
Mazur intersection property. \label{bosmip}
\end{theorem}

The next theorem, due to Todor\v cevi\' c (\cite{tod}, see also \cite[Theorem 4.48]{bos}),
 is the main ingredient of our proof.

\begin{theorem} (MM)
 Every Banach space of density $\om$ has a fundamental biorthogonal system.        \label{todo}
\end{theorem}

For proofs of the above results, additional references and more information we refer to the recent book \cite{bos}.
The main result of our note is following theorem.

\begin{theorem}
 (MM) Let $X$ be an Asplund space of density $\om.$ Then $X$ admits an equivalent norm with the Mazur
intersection property. \label{main}
\end{theorem}

\begin{proof}

Because $X$ is an Asplund space, we have  that $\text{dens}X^*=\om$. (One way of seeing this is to use the
Jayne-Rogers selectors of the duality mapping together with the Bishop-Phelps' theorem, see \cite{DGZ} for
details).
 According to Theorem \ref{bosmip}, it suffices to find a fundamental biorthogonal system of $X^*$, with
$\{(x_{\alpha}^*,x_{\alpha})\}_{\alpha < \om} \subset X^* \times X.$ We claim that if there exists a linear continuous
mapping $T:(X^*,w^*) \rightarrow (c_0(\om),\tau_p)$ with a nonseparable range, such that
\begin{equation}
 T(x^*) = (x^*(y_{\gamma}))_{\gamma \in \Gamma}, \quad \textnormal{for some } \Gamma \textnormal{ uncountable, and }
y_{\gamma}\in S_X, \label{oper}
\end{equation}
then $X^*$ has a fundamental biorthogonal system $\{(x_{\alpha}^*,x_{\alpha})\}_{\alpha < \om} \subset X^* \times X.$
The proof of this claim follows by inspection of Todor\v cevi\' c's argument in \cite{tod} (or \cite[Theorem 4.45]{bos})
which shows the same implication without  $w^*$-topology involved. As a matter of fact, during the construction of a
fundamental system, one only passes to a suitable long subsequence of $\{y_\gamma\}$, which then
becomes the $\{x_\alpha\}\subset X\subset X^{**}$.

It remains to find  $\{y_{\gamma}\}_{\gamma \in \Gamma} \subset S_X$, $\Gamma$ uncountable, that define the operator
in (\ref{oper}).
By Theorem \ref{todo}, there exists a fundamental biorthogonal system in $X$, denoted by
$\{(e_{\gamma}, f_{\gamma})\}_{\gamma < \om} \subset X \times X^* $. We may without loss of generality assume that $\|e_\gamma\|=1$
for all $\gamma<\om$. We shall show that the desired sequence
$\{y_{\gamma}\}_{\gamma \in \Gamma}$ can be found as an uncountable subsequence of $(e_{\gamma})_{\gamma < \om},$ (hence
$\Gamma \subset \om$).

Since $X$ is an Asplund space, there exists a projectional generator $(W,\Phi)$ on $X^*,$ where
$W =  \text{span}\{e_{\gamma}\}_{\gamma < \om}$.  Indeed, the set $W$ is dense
in $X$, and thus it is 1--norming for $X^*$.  By \cite[Proposition 8.2.1]{fab}, there exists a projectional generator
$(X, \Phi)$ on $X^*$, so the restriction of $\Phi$ to $W$ will give the needed mapping.

By  Theorem \ref{subo}, there exists a PRI $\{ P_{\alpha}: \omega \leq \alpha \leq \om \}$ and such that
\begin{equation}
  P_{\alpha} (f_{\gamma}) = \left\{ \begin{array}{l}  f_{\gamma}, \textnormal{ or} \\ 0, \end{array}   \right.
\label{subord}
\end{equation}
for all $\omega \leq \alpha \leq \om$ and $\gamma \in \om ,$ since the  set $G=\{f_\gamma\}_{\gamma<\omega_1}$ countably supports $W.$
We claim that we are allowed to put another additional condition on our PRI, namely we require that

\begin{equation}
  P^*_{\alpha} (e_{\gamma}) = \left\{ \begin{array}{l}  e_{\gamma}, \textnormal{ or} \\ 0, \end{array}   \right.
\label{tri}
\end{equation}
and moreover $P_\alpha^*(e_\gamma)=e_\gamma$ if and only if $P_\alpha(f_\gamma)=f_\gamma$.

This is achieved by entering into the construction of PRI using the PG $(W,\Phi)$.
In fact, reading through the proof of Theorem 3.42 in \cite{bos}, we see that the auxiliary multi-valued mapping
$\Psi: X\to 2^W$ may be without loss of generality chosen so that $e_\gamma\in \Psi(f_\gamma)$ for all $\gamma<\om$.
We also adjust the original $\Phi$ (which in the case of an Asplund space may be chosen a sequence of
Jayne-Rogers' selectors restricted to $W$) by adding  finitely many extra elements
$\{f_{\gamma_i}\}_{i=1}^n$ to the set $\Phi(v)$, for every
$v=\sum_{i=1}^n r_i e_{\gamma_i}, \: r_i \in \rls.$
The adjusted pair $(W, \Phi)$ is again a PG for $X^*$.
 In the case of the first adjustment it is obvious, because
the defining properties of the originally chosen $\Psi$ have not been violated. In the second case, it suffices to note
a simple general fact, that for any given PG $(W,\Phi)$, upon replacing $\Phi$ by any countably valued mapping
$\tilde\Phi$, with the
property that $\Phi(f)\subset\tilde\Phi(f)$, the pair $(W,\tilde\Phi)$ will remain a PG. This is apparent from the
definition of PG, since $\tilde\Phi(B)\supseteq\Phi(B)$ implies $(\tilde\Phi(B))^\perp\subseteq\Phi(B)^\perp$.
Having these extra conditions at hand, it is easy to observe that the set $B_\alpha$ in the proof of Theorem 3.42 in
\cite{bos} equals to the $\kve$--span of the (countable) set $\{e_\lambda\}_{\lambda\in\Lambda_\alpha}$, where
$\Lambda_\alpha=\{\gamma: e_\gamma\in B_\alpha\}$. This implies, using Goldstine's theorem, that
$P^*_\alpha(X^{**})=\clsp \{e_\lambda: \lambda\in\Lambda_\alpha\}$ and also
$(\text{Id}-P^*_\alpha)(X^{**})=\clsp \{e_\lambda: \lambda\notin\Lambda_\alpha\}$.
Thus $P^*_\alpha(e_\gamma)=e_\gamma$ if $\gamma\in\Lambda_\alpha$ and $P^*_\alpha(e_\gamma)=0$ otherwise.
Denote $\Gamma_\alpha=\{\gamma<\om : f_{\gamma} \in (P_{\alpha+1} - P_{\alpha}) X^*\}$.
Let $\Gamma=\{\alpha: \Gamma_\alpha\ne\emptyset\}$.
We claim that $\text{card}\Gamma=\om$. Indeed, $P_\alpha(X^*)$ is a separable Banach space for every $\alpha<\om$,
while $\text{span}\{f_\gamma\}_{\gamma<\om}$ is nonseparable. For convenience we may without loss of generality assume that
$\Gamma=[\omega,\om)$. For every $ \alpha\in\Gamma$ we choose  $\gamma(\alpha)< \om$, such that
$\gamma(\alpha)\in\Gamma_\alpha$.
Since the mapping
$$x^* \mapsto \left(\| (P_{\alpha+1} - P_{\alpha}) (x^*) \|\right)_{\omega \leq \alpha < \om} $$
is into $c_0\left( [\omega,\om)\right)$ (Lemma VI.1.2 in \cite{DGZ}), we have that the mapping
$$T: x^* \mapsto \left(e_{\gamma(\alpha)} ((P_{\alpha+1} - P_{\alpha})(x^*))  \right)_{\alpha \in \Gamma} $$
is into $c_0(\Gamma).$
However, it is a consequence of (\ref{tri}) that $e_{\gamma(\alpha)}=e_{\gamma(\alpha)}\circ (P_{\alpha+1}-P_\alpha)$ on $X^*$.
It follows that
$$ T: x^* \mapsto   \left(x^*(e_{\gamma})\right)_{\gamma \in \Gamma} $$
is a $w^*-\tau_p$ continuous operator from $X^*$ into $c_0(\Gamma)$.
Further, from (\ref{subord}) and the choice of $\Gamma,$ we get that the range of $T$ in not
 separable. This finishes the proof.

\end{proof}

We will conclude the paper by making a few remarks. Note that on the one hand, the dual to our original Asplund space,
$X^*$, admits a PG, so it has  an M-basis, and thus it has a bounded injection into $c_0(\om)$
 (for these facts see \cite{bos}). On the other hand, $X^*$ in general does not admit injections into $c_0(\om)$,
that are continuous in the $w^*-\tau_p$ topology,
since the last condition characterizes WCG spaces $X$, and not every Asplund space
of density $\om$ is WCG. The main point of our construction is that under MM one can actually
 construct a "large" (nonseparable range) operator from $X^*$ into $c_0(\om)$, which is at the
same time $w^*-\tau_p$ continuous.

\end{document}